\documentclass[12pt]{amsart}
\usepackage[lmargin=1in, rmargin=1in, tmargin=1in, bmargin=1in]{geometry} 
\usepackage{graphicx}
\usepackage{amssymb, amsthm}
\usepackage[all]{xy}
\usepackage{comment}
\usepackage[inline]{enumitem}
\usepackage{thm-restate}
\usepackage{verbatim}
\usepackage{xcolor}
\definecolor{darkgreen}{rgb}{0,0.4,0}
\usepackage{tikz}
\usepackage{soul}
\definecolor{darkblue}{rgb}{0,0,0.4}
\usepackage[bookmarks=true, bookmarksopen=true,%
    bookmarksdepth=3,bookmarksopenlevel=2,%
    colorlinks=true,%
    linkcolor=darkblue,%
    citecolor=darkblue,%
    filecolor=darkblue,%
    menucolor=darkblue,%
    urlcolor=darkblue]{hyperref} 
\usepackage[nameinlink]{cleveref}
\usepackage{todonotes}
\usepackage{cite}
\usepackage{mathrsfs}

\hypersetup{colorlinks=true,linkcolor=teal,citecolor=purple} 

\newtheorem{theorem}{Theorem}[section]
\newtheorem{lemma}[theorem]{Lemma}
\newtheorem{proposition}[theorem]{Proposition}

\theoremstyle{definition}

\newtheorem{remark}[theorem]{Remark}

\newtheorem{observation}[theorem]{Observation}

\newcommand{\Z}{{\ensuremath{\mathbb{Z}}}}

\newcommand{\FF}{\ensuremath{\mathbb{F}}}

\newcommand{\HFt}{\underline{\widehat{\mathit{HF}}}}
\newcommand{\HFhat}{\widehat{\mathit{HF}}}
\newcommand{\CFa}{\widehat{\mathit{CF}}}
\newcommand{\tCFa}{\underline{\CFa}}
\newcommand{\tHFa}{\underline{\HFa}}
\newcommand{\HFa}{\HFhat}
\newcommand{\Alg}{\mathcal{A}}
\newcommand{\CFDa}{\widehat{\mathit{CFD}}}
\newcommand{\CFAa}{\widehat{\mathit{CFA}}}
\newcommand{\tCFDa}{\underline{\CFDa}}
\newcommand{\tCFAa}{\underline{\CFAa}}

\newcommand{\Ainf}{A_\infty}
\newcommand{\bdy}{\partial}
\newcommand{\DT}{\boxtimes}
\newcommand{\HF}{\mathit{HF}}

\begin{document}

\title{Detecting Heegaard Floer homology solid tori}

\author{Akram Alishahi}
\thanks{\texttt{AA was partly supported by NSF Grant DMS-2238103}}
\address{Department of Mathematics, University of Georgia, Athens, GA 30602}
\email{\href{mailto:akram.alishahi@uga.edu}{akram.alishahi@uga.edu}}

\author{Tye Lidman}
\address{Department of Mathematics, North Carolina State University, Raleigh, NC 27607}
\email{\href{mailto:tlid@math.ncsu.edu}{tlid@math.ncsu.edu}}
\thanks{\texttt{TL was partly supported by NSF Grant DMS-2105469}}

\author{Robert Lipshitz}
\address{Department of Mathematics, University of Oregon, Eugene, OR
  97403, United States}
\email{\href{mailto:lipshitz@uoregon.edu}{lipshitz@uoregon.edu}}
\thanks{\texttt{RL was partly supported by NSF Grant DMS-2204214}}

\begin{abstract}
    We show that a rational homology solid torus is a Heegaard Floer homology solid torus if and only if it has a Dehn filling with a non-separating 2-sphere. Using this, we characterize Seifert fibered Heegaard Floer solid tori.
\end{abstract}

\maketitle

\section{Introduction}
Bordered Floer homology is an extension of Ozsv\'ath-Szab\'o's Heegaard Floer homology~\cite{OS04:HolomorphicDisks} to 3-manifolds with boundary~\cite{LOT1}. Roughly, the invariant takes as input a compact, connected, oriented 3-manifold with connected boundary, together with a parameterization of the boundary, which can be viewed as a collection of circles giving a basis for the peripheral subgroup, and associates to this data an ($A_\infty$-) module over a certain algebra.  Despite its algebraic complexity, bordered Floer homology has had a variety of applications, especially in the case of 3-manifolds with torus boundary. For instance, it is used in the proof that all rational homology spheres with Sol geometry are L-spaces~\cite[Theorem 2]{BoyerGordonWatson13:LspaceLO}.  The key insight in that proof is that the bordered module associated to the twisted I-bundle over the Klein bottle is invariant under changing the boundary parametrization by Dehn twists along the rational longitude \cite[Proposition 7]{BoyerGordonWatson13:LspaceLO}. (Recall that the rational longitude of a rational homology solid torus $M$ is the unique slope on the boundary which is trivial in $H_1(M;\mathbb{Q})$.)
 
Subsequently, Watson defined a \emph{Heegaard Floer homology solid torus} (HFST) to be a rational homology solid torus for which Dehn twisting along the rational longitude does not change the bordered invariants (see also~\cite{HW23:calculus}).  
 The goal of this paper is the following characterization of HFSTs:

\begin{theorem}\label{thm:main}
    Let $M$ be a rational homology solid torus and let $\lambda$ denote the rational longitude.  Then $M$ is a Heegaard Floer homology solid torus if and only if $M(\lambda)$ contains a non-separating 2-sphere.  
\end{theorem}   

HFSTs have already received considerable interest. In particular, they have a natural interpretation in terms of Hanselman-Rasmussen-Watson's immersed curve formulation of bordered Floer homology, which we will use in the proof of Theorem~\ref{thm:main}:
\begin{proposition}\label{prop:equiv}\cite[Proposition 7.11]{HRW}
    Let $M$ be a rational homology solid torus.  Let $\lambda$ denote the rational longitude of $M$, and fix a dual curve $\mu$ to $\lambda$.  Then the following are equivalent:
    \begin{enumerate}
    \item The manifold $M$ is an HFST. 
    \item The dimension $\dim \HFhat(M(\mu + k \lambda))$ of the Floer homology of the Dehn filling is independent of $k$.
    \item The immersed curve $\HFa(M)$ is supported in a neighborhood $\nu(\lambda)$ of the rational longitude (through the basepoint $z$), after pulling tight.    
    \end{enumerate}
\end{proposition}    

In addition to the solid torus, and the twisted I-bundle over the Klein bottle, familiar examples of HFSTs include the Seifert manifolds $D^2(1/n,-1/n)$. Early examples of HFSTs, like these, have the property that all non-longitudinal fillings are L-spaces. Indeed, Gillespie~\cite{Gillespie} showed that a rational homology solid torus $M$ has both an $S^2 \times S^1$ filling and an L-space filling if and only if all non-longitudinal fillings are L-spaces. 
Similarly, Hanselman-Rasmussen-Watson showed that if one filling of a rational homology solid torus $M$ is an L-space, then $M$ is an HFST if and only if all fillings except along the rational longitude are L-spaces~\cite[Theorem 27]{HRW22:properties}.
Consequently, Theorem~\ref{thm:main} can be viewed as characterizing HFSTs in the case that some (or equivalently every) filling is not an L-space.
There are interesting examples of such HFSTs, such as the one found by Levine-Lidman-Piccirillo~\cite[Remark 9.5]{LLP} (for which filling along the rational longitude produces $S^2\times S^1$, but no filling produces an L-space).

Using Theorem~\ref{thm:main} we can also characterize the HFSTs with Seifert fibered geometry.  
\begin{theorem}\label{thm:seifert}
A Seifert fibered rational homology solid torus $M$ is an HFST if and only if either $M$ has base orbifold a M\"obius band (with or without cone points) or is the Seifert fibered space $D^2(0; p/q, -p/q)$ for some $p/q \neq 0$.  
\end{theorem}


\begin{remark}
It is worth pointing out an analogous picture in the instanton world.  For simplicity, we restrict to integer homology solid tori.  Let $M$ be an integer homology solid torus and let $\iota^*\chi(M)$ denote the image of the $SU(2)$-character variety of $M$ in the pillowcase.  Then $M$ is the complement of a knot in a 3-manifold with a non-separating $S^2$ if and only if $\iota^*\chi(M)$ has no intersection with the top seam of the pillowcase, the flat connections on $\partial M$ with holonomy $-\mathit{Id}$ around the homological longitude (see, e.g., \cite[Proposition 5.1]{Zentner:SL2} and \cite[Theorem 1.6]{KM:ICM}).  
\end{remark}

This paper is organized as follows. Section~\ref{sec:background} collects the results we need about bordered Heegaard Floer homology and its immersed curve formulation. Section~\ref{sec:knots-in-red-mflds} uses the surgery exact triangle with twisted coefficients to prove that knot complements in reducible manifolds give Heegaard Floer solid tori. Section~\ref{sec:all-HFSTs-are-these} proves the other direction of Theorem~\ref{thm:main}, that all Heegaard Floer solid tori arise as such knot complements.  Finally, Section~\ref{sec:seifert} characterizes Seifert HFSTs. 

\subsection*{Acknowledgments.} The second author thanks Liam Watson for many conversations about HFSTs over the years. 

\section{Background}\label{sec:background}
We collect the facts we will use about bordered Heegaard Floer homology for 3-manifolds with torus boundary and its reformulation in terms of immersed curves.

\subsection{Classical bordered Floer homology}
Bordered Floer homology associates to the torus $T^2$ an algebra $\Alg$ over $\FF_2$ containing a distinguished pair of orthogonal idempotents $\iota_0$ and $\iota_1$. There are six other basis elements over $\FF_2$, with names like $\rho_{12}$~\cite[Section 11.1]{LOT1}. To a 3-manifold $M$ and a diffeomorphism $\bdy M\cong T^2$, bordered Floer homology associates a twisted complex (\emph{type $D$ structure}) $\CFDa(M)$ and an $\Ainf$-module $\CFAa(M)$, each well-defined up to homotopy equivalence. The \emph{pairing theorem} states that, given 3-manifolds $M_1$ and $M_2$ with identifications $\bdy M_1\cong\bdy (-M_2)\cong T^2$, there is a chain homotopy equivalence $\CFa(M_1\cup_{T^2} M_2)\simeq \CFAa(M_1)\DT\CFDa(M_2)$, where $\DT$ is an appropriate version of the tensor product, so long as one of $\CFAa(M_1)$ or $\CFDa(M_2)$ is bounded~\cite[Theorem 1.3]{LOT1}.
(The minus signs denote orientation reverses, and we may drop them in the future.) The notion of \emph{boundedness} is a kind of upper-triangularity condition, as is usually required for twisted complexes (e.g.,~\cite[Section (3l)]{SeidelBook}), and corresponds to a property of Heegaard diagrams called admissibility~\cite[Section 4.2.2]{OS04:HolomorphicDisks},~\cite[Lemmas 6.5 and 7.7]{LOT1}.

For example, if $M$ is the solid torus then the invariant $\CFAa(M)$ of $M$ with a particular parameterization of the boundary, the \emph{$0$-framing}, can be described, say, as either
\[
\FF_2\langle n \rangle,\ n=m_3(n,\rho_2,\rho_1)=m_4(n,\rho_2,\rho_{12},\rho_1)=m_5(n,\rho_2,\rho_{12},\rho_{12},\rho_1)=\cdots
\]
or
\[
\FF_2\langle n,p,q\rangle,\ m_1(p)=q,\ m_2(p,\rho_1)=n,\ m_2(n,\rho_2)=q,\ m_2(p,\rho_{12})=q.
\]
Here, we have specified all non-vanishing actions by elements of the form $\rho_i$; the actions by the idempotents are determined by $n\iota_1=n$, $p\iota_0=p$, $q\iota_0=q$. The second of these models for $\CFAa(M)$ is bounded; the first is not. Denote the second of these modules by $\mathcal{S}$.

Not every $\Alg$-module arises as the invariant of a 3-manifold with boundary. In particular, the invariants $\CFDa(M)$ associated to 3-manifolds with torus boundary have an extra property~\cite[Proposition 11.30]{LOT1}, which is now called being \emph{extendable}~\cite{HRW}. (This property can be seen as a shadow of the existence of a bordered extension of $\mathit{HF}^-$~\cite[Proposition 7.16]{LOT:torus-mod}.)

Given a closed 3-manifold $Y$ and a module $R$ over the group ring $\FF_2[H_2(Y)]$, there is a refinement $\tCFa(Y;R)$ of the Heegaard Floer homology of $Y$ with \emph{twisted coefficients $R$}~\cite[Section 8]{OS04:HolDiskProperties}. The case of interest to us is when $H_2(Y)\cong\Z$, so $R$ is a module over the Laurent polynomial ring $\FF_2[t,t^{-1}]$. In particular, if $R$ is any field so that $\FF_2[t,t^{-1}]\hookrightarrow R$, then the homology $\tHFa(Y;R)$ vanishes if and only if $Y$ contains a homologically essential embedded 2-sphere (or, equivalently, has an $S^2\times S^1$-summand)~\cite[Theorem 7.11]{AL19:incompressible}.

There is also a twisted coefficient version of bordered Floer homology. Given a module $R$ over $\FF_2[H_2(M,\bdy M)]$ there are twisted invariants $\tCFAa(M;R)$ and $\tCFDa(M;R)$~\cite[Sections 6.4 and 7.4]{LOT1}. For example, for the $0$-framed solid tori above, if we identify $H_2(S^1\times D^2,T^2)=\Z$ so that $\FF_2[H_2(M,\bdy M)]=\FF_2[t,t^{-1}]$, and take $R=\FF_2[[t,t^{-1}]$ to be the Laurent series, then the twisted versions of the modules above are, say, 
\[
R\langle n\rangle,\ m_3(n,\rho_2,\rho_1)=tn,\ m_4(n,\rho_2,\rho_{12},\rho_1)=t^2n,\ \cdots
\]
(unbounded) or
\[
R\langle n,p,q\rangle,\ m_1(p)=q,\ m_2(p,\rho_1)=n,\ m_2(n,\rho_2)=tq,\ m_2(p,\rho_{12})=tq
\]
(bounded). Denote the second of these modules by $\underline{\mathcal{S}}$.

The twisted coefficient bordered invariants again satisfy a pairing theorem~\cite[Theorem 9.44]{LOT1}. The case of interest to us is when $M_1$ and $M_2$ are rational homology solid tori, $H_2(M_1\cup_{T^2}M_2)\cong \Z$, and $R=\FF_2[[t,t^{-1}]$, viewed as an algebra over $\FF_2[H_2(M_1,\bdy M_1)]$ and $\FF_2[H_2(M_1\cup_{T^2}M_2)]$ in the obvious way. In this case,
\begin{equation}\label{eq:twist-pair-special-case}
\tCFa(M_1\cup_{T^2}M_2;R)\simeq \tCFAa(M_1;R)\DT\CFDa(M_2).
\end{equation}

We conclude with two lemmas about vanishing of twisted Floer groups.

\begin{lemma}\label{lem:all-in-one}
    Let $P$ be a finitely generated type $D$ structure over $\Alg$ so that $\iota_1P=0$ (i.e., all generators of $P$ lie over the idempotent $\iota_0$). Then $\underline{\mathcal{S}}\DT P$ has trivial homology.
\end{lemma}
\begin{proof}
    There is a filtration on $\underline{\mathcal{S}}\DT_{\Alg} P$ by $t$-powers, i.e., induced by the degree filtration on $\FF_2[[t,t^{-1}]$, and $\underline{\mathcal{S}}\DT_{\Alg} P$ is complete with respect to this filtration. The $E^1$-page of the associated spectral sequence is isomorphic, as an $\FF$-vector space, to $\iota_1P$: generators of the form $p\otimes x$ and $q\otimes x$ cancel in pairs. The result follows.
\end{proof}

\begin{lemma}\label{lem:untwist-to-twist}
    If $P$ is a type $D$ structure so that $\mathcal{S}\DT P$ has trivial homology, then $\underline{\mathcal{S}}\DT P$ also has trivial homology.
\end{lemma}
\begin{proof}
    Let $\mathscr{S}$ be the $\Ainf$-module with the same generators and operations as $\underline{\mathcal{S}}$, but defined over $\FF_2[t,t^{-1}]$. So, $\underline{\mathcal{S}}=\mathscr{S}\otimes_{\FF_2[t,t^{-1}]}\FF_2[[t,t^{-1}]$ and $\mathcal{S}=\mathscr{S}\otimes_{\FF_2[t,t^{-1}]}\FF_2$, and corresponding statements hold after tensoring with $P$. Since $\FF_2[t,t^{-1}]$ is a PID, the Universal Coefficient Theorem applies, so the dimension of $H_*(\underline{\mathcal{S}}\DT P)$ (over $\FF_2[[t,t^{-1}]$) is the rank of $H_*(\mathscr{S}\DT P)$, and the dimension of $\mathcal{S}\DT P$ is at least this large. The result follows.
\end{proof}

\subsection{Bordered Floer homology via immersed curves}
Hanselman, Rasmussen, and Watson gave a reformulation of bordered Floer homology for 3-manifolds with torus boundary. First, they showed that homotopy equivalence classes of extendable type $D$ structures over $\Alg$ are in bijection with what they call weak equivalence classes of admissible train tracks in the punctured torus, perhaps decorated with local systems of $\FF_2$-modules~\cite[Proposition 3.6]{HRW}. Immersed (closed) 1-manifolds in the punctured torus are a special case of train tracks, and they then show that any admissible train track with local system is weakly equivalent to an immersed 1-manifold with local system~\cite[Theorem 1.5]{HRW}.
(The immersed curves that arise are \emph{unobstructed}~\cite[Theorem 4.11]{HRW}, meaning that they lift to embedded loops in appropriate covers~\cite[Definition~4.1]{HRW}.)

Their correspondence depends on an extra marking of the torus, by a pair of embedded circles in $T^2$ generating $\pi_1(T^2)$ and intersecting transversely in one point. Equivalently, it depends on an identification of $T^2$ as a quotient of $[0,1]^2$ by identifying opposite sides. An important property of their construction is the following:
\begin{observation}\label{obs:all-in-one}
    Let $T$ be an admissible train track in $T^2$ which is disjoint from the circle $[0,1]\times\{0\}$. Then the corresponding type $D$ structure $P$ satisfies $\iota_1P=0$.
\end{observation}

They give two versions of the pairing theorem~\cite[Theorems 2.2 and 4.11]{HRW}. For general train tracks, one embeds one train track so that outside $[1/2,1]^2$ it consists of horizontal and vertical segments. One rotates the other train track by $\pi/2$ and embeds it so that outside $[0,1/2]^2$ it, too, consists of horizontal and vertical train tracks. Then, one computes a version of Lagrangian intersection Floer homology of the pair (which is combinatorial, by the Riemann mapping theorem). For immersed curves, one rotates one immersed curve by $\pi/2$ and computes Lagrangian intersection Floer homology; the embedding is not important. In both cases, however, one needs an extra condition, \emph{admissibility}, to systematically guarantee finiteness of sums. For train tracks, they formulate admissibility as the condition that there is no immersion of an annulus with one boundary component on each train track~\cite[Before Proposition 3.7]{HRW}. For immersed curves, admissibility is the condition that every periodic domain has positive and negative multiplicities~\cite[Definition 4.7]{HRW}. The latter condition is vacuously true unless there is a pair of components which are \emph{commensurable}, i.e., so that the corresponding elements of the fundamental group of the punctured torus have a common multiple.

\section{Knot complements in reducible manifolds}\label{sec:knots-in-red-mflds}
In this section, we prove the easy direction of Theorem~\ref{thm:main}:
\begin{proposition}
Let $Y$ be a closed 3-manifold with $b_1 = 1$ which contains a non-separating 2-sphere.  Let $K$ be a knot in $Y$ which is infinite order in $H_1(Y)$.  Then the complement of $K$ is a Heegaard Floer homology solid torus.    
\end{proposition}
\begin{proof}
Let $M$ denote the complement of $K$.  It is straightforward to verify that $M$ is a rational homology solid torus.  Parametrize $\partial M$ by choosing the rational longitude $\lambda$ and a dual curve $\mu$. Consider the surgery exact sequence with Laurent series coefficients   
\[
\cdots\to \HFt(M(\mu);\FF_2[[t,t^{-1}]) \to \HFt(M(\mu + \lambda);\FF_2[[t,t^{-1}]) \to \HFt(M(\lambda);\FF_2[[t,t^{-1}]) \to \cdots 
\]
(e.g.,~\cite[Theorem 9.21]{OS04:HolDiskProperties} or~\cite[Theorem 3.1]{AiPeters10}).
Since $M(\lambda)$ contains a non-separating 2-sphere,
$\dim \HFt(M(\lambda);\FF_2[[t,t^{-1}])=0$.  Since $M(\mu)$ and $M(\mu + \lambda)$ are rational homology spheres, the dimensions of their twisted Heegaard Floer homologies over the field $\FF_2[[t,t^{-1}]$ are the same as the dimensions of their untwisted Heegaard Floer homologies over $\FF_2$. Hence, $\dim \HFhat(M(\mu)) = \dim \HFhat(M(\mu + \lambda))$.  Note that $\mu + \lambda$ is also a dual curve for $\lambda$, so we can apply the same argument to see that $\dim \HFhat(M(\mu + \lambda)) = \dim \HFhat(M(\mu + 2\lambda))$, and so on: $\dim \HFhat(M(\mu)) = \dim \HFhat(M(\mu + k \lambda))$ for all $k$.  The result now follows from Proposition~\ref{prop:equiv}.  
\end{proof}

\section{All HFSTs arise as such}\label{sec:all-HFSTs-are-these}
In this section, we prove the other easy direction of Theorem~\ref{thm:main}.  By Proposition~\ref{prop:equiv}, it suffices to prove:
\begin{proposition}\label{prop:other-easy-dir}
Let $M$ be a rational homology solid torus such that $\HFa(M)$ is supported in a neighborhood of $\lambda$ after pulling tight.  Then $\HFt(M(\lambda);\FF_2[[t,t^{-1}]) = 0$, and hence $M(\lambda)$ has a non-separating 2-sphere.   
\end{proposition} 
\begin{proof}
    If $\HFt(M(\lambda);\FF_2[[t,t^{-1}]) = 0$, then $M(\lambda)$ contains a non-separating 2-sphere~\cite[Theorem 1.1]{AL19:incompressible}. So, by Formula~\eqref{eq:twist-pair-special-case}, it suffices to prove that 
    \[
    H_*\bigl(\underline{\mathcal{S}}\DT\CFDa(M)\bigr)=0.
    \]
    Let $C_1,\dots,C_k$ be the components of the immersed curve $\HFa(M)$ (together with their local systems). For each component $C_i$ there is a corresponding (chain homotopy equivalence class of) type $D$ structure, which we denote $\CFDa(C_i)$, and $\CFDa(M)=\bigoplus_{i=1}^k\CFDa(C_i)$. So, it suffices to prove that, for each $i$, $H_*\bigl(\underline{\mathcal{S}}\DT\CFDa(C_i)\bigr)=0$. There are two cases. We continue to denote the rational longitude by $\lambda$.

    \emph{Case 1.} $C_i$ is not homotopic to $\lambda^j$ for any $j$. By hypothesis, $C_i$ is regularly homotopic to a curve which is disjoint from $\lambda$. The pair $(\lambda, C_i)$ is admissible (because any pair of curves which are not commensurable are admissible), so $H_*(\mathcal{S}\DT\CFDa(C_i))\cong \HF(\lambda,C_i)=0$. Hence, the result follows from Lemma~\ref{lem:untwist-to-twist}. (See Figure~\ref{fig:admissible-case} for an example.)

    \emph{Case 2.} $C_i$ is homotopic to $\lambda^j$. We claim that $C_i$ is regularly homotopic to a curve that lies in a neighborhood of $\lambda$. Let $\Sigma$ be the cover of the punctured torus corresponding to the subgroup $\langle \lambda^j\rangle\subset \pi_1(T^2\setminus \{p\})$ of the fundamental group of the punctured torus (with respect to some basepoint on $\lambda$). So, $\Sigma$ is topologically a cylinder. By hypothesis, $C_i$ lifts to a loop in $\Sigma$ representing a generator of $\pi_1(\Sigma)=\Z$ and, in fact, this loop is embedded~\cite[Definition 4.1]{HRW}.
    The preimage of $\lambda$ is also an embedded loop representing a generator of $\pi_1$. Any two such loops in the cylinder are regularly homotopic. Projecting that regular homotopy to the punctured torus gives (and stopping just before the end) gives the desired regular homotopy of $C_i$ into a neighborhood of $\lambda$. Now, by Observation~\ref{obs:all-in-one}, the corresponding type $D$ structure $\CFDa(C_i)$ has $\iota_1\CFDa(C_i)=0$. Hence, the result follows from Lemma~\ref{lem:all-in-one}. (See Figure~\ref{fig:inadmissible-case} for an example.)
\end{proof}

\begin{figure}
    \centering
    \begin{tikzpicture}[scale=4]
        \draw[red] (0,0) -- (1,0) -- (1,1) -- (0,1) -- (0,0);
        \draw[darkgreen, dashed] (0,.4) -- (1,.4);
        \draw[blue] (0,.6) to[out=0,in=270] (.33,1);
        \draw[blue] (.66,1) to[out=270,in=180] (1,.8);
        \draw[blue] (0,.8) to[out=0,in=180] (1,.6);
        \draw[blue] (0,.2) to[out=0,in=90] (.33,0);
        \draw[blue] (.66,0) to[out=90,in=180] (1,.2);
    \end{tikzpicture}
    \caption{\textbf{An example of the first case in the proof of Proposition~\ref{prop:other-easy-dir}}. The longitude is \textcolor{darkgreen}{dashed} and the immersed curve is \textcolor{blue}{solid}; this pair is admissible.}
    \label{fig:admissible-case}
\end{figure}
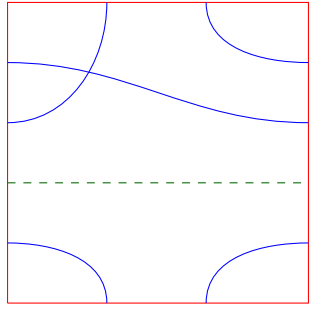

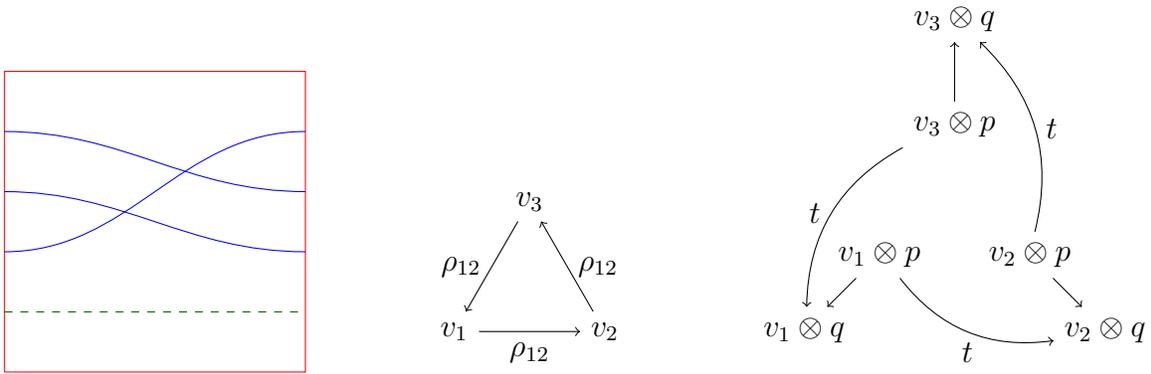
\begin{figure}
    \centering
    \begin{tikzpicture}[scale=4]
        \draw[red] (0,0) -- (1,0) -- (1,1) -- (0,1) -- (0,0);
        \draw[darkgreen, dashed] (0,.2) -- (1,.2);
        \draw[blue] (0,.4) to[out=0,in=180] (1,.8);
        \draw[blue] (0,.6) to[out=0,in=180] (1,.4);
        \draw[blue] (0,.8) to[out=0,in=180] (1,.6);
    \end{tikzpicture}\qquad\qquad
    \begin{tikzpicture}[scale=2]
        \node at (0,0) (bl) {$v_1$};
        \node at (1,0) (br) {$v_2$};
        \node at (.5,.866) (tc) {$v_3$};
        \draw[->] (bl) to node[below]{$\rho_{12}$} (br);
        \draw[->] (br) to node[right]{$\rho_{12}$} (tc);
        \draw[->] (tc) to node[left]{$\rho_{12}$} (bl);
    \end{tikzpicture}\qquad\qquad
    \begin{tikzpicture}[scale=2]
        \node at (0,0) (blp) {$v_1\otimes p$};
        \node at (-.5,-.5) (blq) {$v_1\otimes q$};
        \node at (1,0) (brp) {$v_2\otimes p$};
        \node at (1.5,-.5) (brq) {$v_2\otimes q$};
        \node at (.5,.866) (tcp) {$v_3\otimes p$};
        \node at (.5,1.57) (tcq) {$v_3\otimes q$};
        \draw[->, bend right=30] (blp) to node[below]{$t$} (brq);
        \draw[->, bend right=30] (brp) to node[right]{$t$} (tcq);
        \draw[->, bend right=30] (tcp) to node[left]{$t$} (blq);
        \draw[->] (blp) to (blq);
        \draw[->] (brp) to (brq);
        \draw[->] (tcp) to (tcq);
    \end{tikzpicture}
    \caption{\textbf{An example of the second case in the proof of Proposition~\ref{prop:other-easy-dir}}. Left: the immersed curve invariant, with conventions as in Figure~\ref{fig:admissible-case}. This pair is not admissible. Center: the corresponding bordered invariant. Right: the tensor product with $\underline{\mathcal{S}}$.}
    \label{fig:inadmissible-case}
\end{figure}

\begin{remark}
An irreducible rational homology solid torus $M$ is a solid torus if and only if every component of the immersed curve $\HFa(M)$ is homotopic to $\lambda^j$ for some $j$. This follows from the fact that $M$ is a solid torus if and only if $\lambda$ bounds a disk in $M$ and $\lambda$ compresses if and only if $\CFDa(M))\simeq P$ such that $\iota_1P=0$ (see \cite[Proposition 7.13]{HRW}). 
\end{remark}

\section{Seifert fibered HFSTs}\label{sec:seifert}
\begin{proof}[Proof of Theorem~\ref{thm:seifert}]
Let $M$ be a Seifert fibered HFST with rational longitude $\lambda$.  
So, $M$ is a circle bundle over a 2-dimensional orbifold $F(p_1,\ldots, p_k)$ which we write to mean that the underlying 2-manifold is $F$ and there are $k$ cone points of orders $p_1,\ldots, p_k \geq 2$.  (If we had some $p_j = 1$, then the fiber over that cone point would just be a regular fiber, so we can throw these out.)  Since $M$ has one torus boundary component, $F$ has a single boundary component.  Because $b_1(M) = 1$ and $\pi_1(M)$ surjects onto the orbifold fundamental group of $F$, which in turn surjects onto $\pi_1(F)$, we see that $F$ must be either a disk $D^2$ or a M\"obius band $N$.  

First consider the case that $F$ is a disk.  If $k = 0$ or $k = 1$, then $M$ is a solid torus, which is an HFST.  Therefore, we consider the case that $k \geq 2$.  Let $\gamma$ denote the slope of a Seifert fiber on $\partial M$ and $\lambda$ the rational longitude.  We claim that $\lambda \neq \gamma$.  If $\lambda = \gamma$, then $M(\gamma)$ is a connected sum of lens spaces by \cite[Proposition 2(b)]{Heil}.  This contradicts the fact that filling the rational longitude gives $b_1 = 1$.  Therefore, $\lambda \neq \gamma$, so we can extend the Seifert fibered structure of $M$ over $M(\lambda)$ so that the core of the Dehn filling becomes a singular fiber of order $\Delta(\gamma, \lambda)$.  In particular, $M(\lambda)$ is a Seifert fibered space of the form $S^2(p_1,\ldots, p_k, \Delta(\gamma, \lambda))$.  The only reducible Seifert fibered spaces are $RP^3 \# RP^3$ and $S^2 \times S^1$ (see, e.g., \cite[Proposition 1.12]{Hatcher}), so we will determine when $M(\lambda)$ is $S^2 \times S^1$ (and we do not need to consider the sum of $S^2 \times S^1$ with any other 3-manifolds).  While $S^2 \times S^1$ has infinitely many different Seifert structures, they all are of the form $S^2(0; p/q, -p/q)$ and hence have at most two singular fibers (see, e.g., \cite[Theorem 2.3]{Hatcher}).  Since we assumed $k \geq 2$, we must have $k = 2$ and $\Delta(\gamma, \lambda) = 1$.   In other words, the core of the Dehn filling must be a regular fiber.  Hence, $M$ is obtained by removing a regular fiber from the Seifert fibered space $S^2(0; p/q, -p/q)$, which gives $D^2(0; p/q, -p/q)$.    

Now consider the case that $F$ is a M\"obius band.  We want to show that $M(\gamma)$ always contains a non-separating 2-sphere, in the absence or presence of singular fibers.  (This will imply that $\lambda = \gamma$, since only one Dehn filling of $M$ has $b_1 = 1$.)  This is well-known, but we include a proof for completeness.  Consider $\eta$ the core curve of $F$, chosen to be disjoint from any cone points in the base orbifold, and an arc $\alpha$ in $F$ which is the generator of $H_1(F, \partial F)$ and geometrically dual to $\alpha$.  Then the preimage of $\alpha$ in $M$ is an annulus $A$ properly embedded in $M$, with boundary two copies of $\gamma$, and there is a section $\sigma$ of the circle bundle over $\eta$ which intersects $A$ in one point.  Hence, in $M(\gamma)$ we can cap off the two boundary components of $A$ with disjoint disks (coming from the surgery solid tori, and disjoint from $\sigma$), and hence we have a 2-sphere in $M(\gamma)$ which intersects $\sigma$ once.  Hence, this 2-sphere is non-separating.  
\end{proof}

\bibliographystyle{hamsalpha}
\bibliography{refs}

\end{document}